\documentclass[11pt,a4paper]{amsart}
\usepackage{amssymb}
\usepackage{amsmath}
\usepackage{amsfonts}
\usepackage[english]{babel}
\usepackage[T1]{fontenc}
\usepackage[latin1]{inputenc}
\usepackage{amsthm}
\usepackage[all]{xy}
\xyoption{all}
\usepackage{bbm}
\usepackage{hyperref}
\usepackage{fixltx2e}
\usepackage{bbold}
\usepackage{dsfont}
\usepackage{verbatim}
\usepackage{graphicx}
\usepackage{tikz,tikz-cd}
\usepackage[numbers]{natbib}
\usepackage[shortlabels]{enumitem}
\usepackage{mathrsfs}
\usepackage{adjustbox}
\usetikzlibrary{matrix, arrows}

\newcommand{\N}{\mathbb{N}}

\newcommand{\K}{\mathcal{K}}
\newcommand{\Nu}{\mathcal{N}}
\newcommand{\A}{\mathcal{A}}
\newcommand{\I}{\mathcal{I}}
\newcommand{\J}{\mathcal{J}}
\newcommand{\QN}{\mathcal{QN}}
\newcommand{\CA}{\mathcal{CA}}

\newcommand{\SK}{\mathcal{SK}}
\newcommand{\F}{\mathcal{F}}
\newcommand{\V}{\Vert\cdot\Vert}
\newcommand{\VI}{\Vert\cdot\Vert_\mathcal I}

\theoremstyle{definition}
\newtheorem{theorem}{Theorem}[section]
\newtheorem{proposition}[theorem]{Proposition}
\newtheorem{lemma}[theorem]{Lemma}

\newtheorem{question}[theorem]{Question}

\newtheorem{remark}[theorem]{Remark}

\newtheorem{eexample}[theorem]{Example}
\newtheorem{fact}[theorem]{Fact}
\theoremstyle{plain}
\numberwithin{equation}{section}

\title[On approximation properties]{On approximation properties related to unconditionally $p$-compact operators and Sinha-Karn $p$-compact operators}
\author[Henrik Wirzenius]{Henrik Wirzenius}
\address{Henrik Wirzenius: Department of Mathematics and Statistics, Box 68, Pietari Kalmin katu 5,
FI-00014 University of Helsinki, Finland}
\email{henrik.wirzenius@helsinki.fi}
\thanks{This work was supported by the Magnus Ehrnrooth Foundation and the Swedish Cultural Foundation in Finland}
\begin{document}

\begin{abstract}
We establish new results on the $\mathcal I$-approximation property for the Banach operator ideal $\mathcal I=\mathcal{K}_{up}$ of the unconditionally $p$-compact operators in the case of $1\le p<2$. As a consequence of our results, we provide a negative answer for the case $p=1$ of a problem posed by J.M. Kim (2017). Namely, the $\mathcal K_{u1}$-approximation property implies neither the $\mathcal{SK}_1$-approximation property nor the (classical) approximation property; and the $\mathcal{SK}_1$-approximation property implies neither the $\mathcal{K}_{u1}$-approximation property nor the approximation property. Here $\mathcal{SK}_p$ denotes the $p$-compact operators of Sinha and Karn for $p\ge 1$. We also show for all $2<p,q<\infty$ that there is a closed subspace $X\subset\ell^q$ that fails the $\mathcal{SK}_r$-approximation property for all $r\ge p$.
\end{abstract}

\maketitle

\section{Introduction}
Let $\I=(\I,\VI)$ be a Banach operator ideal. Following Delgado et al. \cite{DPS10_2} and Oja \cite{Oja12}, we say that the Banach space $X$ has the $\I$-approximation property ($\I$-AP in short) if the space $\F(Y,X)$ of all bounded finite-rank operators $Y\to X$ is $\VI$-dense in $\I(Y,X)$ for every Banach space $Y$. The $\I$-AP has recently been studied in several papers for various Banach operator ideals $\I$, see e.g. \cite{DPS10_2, Oja12, LT13, Kim19, Kim20, KLT21} and their references. Moreover, results about special instances of the $\I$-AP have been used in the recent studies \cite{TW22, Wir} of radical quotient algebras of Banach operator ideals.
\smallskip

The study of the $\K_{up}$-AP was initiated in \cite{Kim15} for the Banach operator ideal $\K_{up}$ of the unconditionally $p$-compact operators \cite{Kim14}, where $p\in[1,\infty)$ is a real number. It is known that if the Banach space $X$ has the classical approximation property (AP in short), then $X$ has the $\K_{up}$-AP for all $1\le p<\infty$. Moreover, due to results of J.M. Kim \cite{Kim15,Kim17_2} and Lassalle and Turco \cite{LT17}, it holds for all $1\le p<\infty$ that $X$ has the $\K_{up}$-AP whenever the dual space $X^*$ has the $\SK_{p}$-AP, and interchangeably, $X$ has the $\SK_p$-AP whenever $X^*$ has the $\K_{up}$-AP. Here $\SK_p$ denotes the Banach operator ideal of the Sinha-Karn $p$-compact operators introduced by Sinha and Karn \cite{SK02}. These duality results connect the $\K_{up}$-AP to the theory of the Sinha-Karn $p$-compact operators, which is a class of operators that has received a lot of attention over the past 20 years, see e.g. \cite{SK02,CK10,DPS10,PD11,GLT12,Pietsch14}. 
\smallskip

The main purpose of the present paper is to establish new results on the $\K_{up}$-AP in the case of $1\le p<2$. We approach the $\K_{up}$-AP mainly through the characterisation $\K_{up}=\K_{p'}^{sur}$ due to Mu\~{n}oz et al. \cite{MOP15} (also Fourie \cite{Fourie18}), where $\K_{p'}^{sur}$ denotes the surjective hull of the Banach operator ideal $\K_{p'}$ of the (classical) $p'$-compact operators, and $p'$ is the dual exponent of $p$. By considering $\K_{up}$ through this characterisation, many of the classical compact factorisation results related to $p$-compact operators become accessible in a natural way. This allows for new insights and also to recover known facts on the $\K_{up}$-AP. 
\smallskip

Our main results are established in Section 3. In particular, we establish the following monotonicity behaviour of the $\K_{up}$-AP in the range $[1,2)$: if the Banach space $X$ has the $\K_{up}$-AP, then $X$ has the $\K_{uq}$-AP whenever $1\le p<q<2$. Note that this monotonicity behaviour does not extend to the case of $1\le p< 2<q<\infty$, since in that case there are Banach spaces with the $\K_{up}$-AP that fail the $\K_{uq}$-AP (Example \ref{example}). Here we have omitted the trivial cases $p=2$ or $q=2$ since every Banach space has the $\K_{u2}$-AP. Moreover, by applying a factorisation result of Pisier \cite{Pisier80} and a result of John \cite{Jo} on the AP, we show that the Banach space $X$ has the $\K_{up}$-AP for all $1\le p<2$ whenever $X$ has cotype 2. 
\smallskip

As a consequence of our results, we show that there exist reflexive Banach spaces that have the $\K_{u1}$-AP but fail the $\SK_1$-AP. For such a space $X$, the dual $X^*$ has the $\SK_1$-AP but fails the $\K_{u1}$-AP. This answers the case $p=1$ of questions posed by J.M. Kim \cite[Problem 1]{Kim17_2} on the relationship between the AP, the $\SK_p$-AP and the $\K_{up}$-AP. Namely, the $\K_{u1}$-AP implies neither the $\SK_1$-AP nor the AP, and moreover, the $\SK_1$-AP implies neither the $\K_{u1}$-AP nor the AP. In addition, we show that there exist Banach spaces that have the $\mathcal W_1$-AP  but fail the AP, where $\mathcal W_1$ denotes the weakly 1-compact operators of Sinha and Karn \cite{SK02}. This answers a query of Kim posed in \cite{Kim20}.
\smallskip

In Section \ref{section4} we exhibit for all $2<p,q<\infty$ a closed subspace $X\subset\ell^q$ that fails the $\SK_r$-AP for all $r\ge p$. This is essentially the closed subspace $E\subset\ell^q$ constructed by Davie \cite{Da73,Da75} that fails the AP, and it complements a result of Y.S. Choi and J.M. Kim \cite{CK10} in which they established that a variant of the Davie space $E\subset\ell^q$ fails the $p$-approximation property whenever $p>2$ and $q>2p/(p-2)$. We also draw attention to some remaining questions related to the connections between the AP, the $\K_{up}$-AP and the $\SK_p$-AP for $2<p<\infty$.

\section{Banach operator ideals and approximation properties}\label{section2}

In this section we recall relevant definitions, notation and concepts related to Banach operator ideals and the corresponding approximation properties. Moreover, we establish a reverse monotonicity property for the classes $\K_{up}$ of unconditionally $p$-compact operators in the range $1\le p\le 2$, which will be used in Section 3. We also provide a characterisation of the $\K_{up}$-AP for dual Banach spaces which complements similar characterisations of the $\SK_p$-AP and the uniform $\SK_p$-AP obtained in \cite{DPS10_2} and \cite{DOPS09}, respectively.

Throughout the paper we consider Banach spaces over the same, either real or complex, scalar field $\mathbb K$. The closed unit ball of any Banach space $X$ is denoted by $B_X$ and the dual exponent of the real number $p\in[1,\infty)$ is denoted by $p'$ (i.e., $1/p+1/p'=1$).

\subsection{Banach operator ideals}
Let $\mathcal L(X,Y)$ be the space of all bounded linear operators $X\to Y$ for Banach spaces $X$ and $Y$. We say that $\I=(\I,\VI)$ is a \emph{Banach operator ideal} if $\I$ is a complete normed operator ideal in the sense of Pietsch \cite{Pietsch80}. Recall that the following conditions hold for all Banach spaces $X$ and $Y$:
\begin{enumerate}
\item[(B1)] The component $\I(X,Y)$ is a linear subspace of $\mathcal L(X,Y)$, and $\VI$ is a complete norm in $\I(X,Y)$ such that $||T||\le ||T||_\I$ for all $T\in\I(X,Y)$. 
\item[(B2)] The bounded rank-one operator $x^*\otimes y\in\mathcal I(X,Y)$ and \[||x^*\otimes y||_\I=||x^*||\cdot||y||\] 
for all $x^*\in X^*$ and $y\in Y$.
\item[(B3)] For all Banach spaces $E$ and $F$, and all operators $S\in\mathcal L(E,X)$, $T\in\I(X,Y)$ and $U\in\mathcal L(Y,F)$, one has $UTS\in \I(E,F)$. Moreover, \[||UTS||_\I\leq ||U||\cdot ||S||\cdot ||T||_\I.\]
\end{enumerate}
Above $x^*\otimes y: x\mapsto x^*(x)y$ and $\V$ is the uniform operator norm. The first part of (B3) will be referred to as the \emph{operator ideal property}. We refer to the monographs \cite{Pietsch80, DF93} and the survey \cite{DJP} for comprehensive sources on Banach operator ideals.
\smallskip

Following Pietsch \cite[6.7.1]{Pietsch80}, the inclusion $\I\subset\J$ of two Banach operator ideals $\I=(\I,\VI)$ and $\J=(\J,\V_{\J})$ means that $\I(X,Y)\subset \mathcal J(X,Y)$ for all Banach spaces $X$ and $Y$, and $||\cdot||_\J\le \VI$. Moreover, the identity $\I=\J$ means that $\I\subset\J$ and $\J\subset\I$, that is, $\I(X,Y)=\mathcal J(X,Y)$ for all Banach spaces $X$ and $Y$, and $\VI=||\cdot||_{\J}$.
\smallskip

We proceed by recalling relevant properties of the Sinha-Karn $p$-compact operators and the unconditionally $p$-compact operators. Towards the definitions of these classes of operators, let $p\in[1,\infty)$ be fixed and let $X$ and $Y$ be arbitrary Banach spaces. We denote by $\ell^p_w(X)$ the space of all weakly $p$-summable sequences in $X$, and by $\ell_s^p(X)$ that of all strongly $p$-summable sequences in $X$. Recall that $\ell^p_w(X)$ and $\ell^p_s(X)$ are Banach spaces when equipped with the following natural norms:
\begin{align*}&||(x_k)||_{p,w}=\sup_{x^*\in B_{X^*}}\big(\sum_{k=1}^\infty |x^*(x_k)|^p\big)^{1/p},\quad (x_k)\in\ell_w^p(X),\\
&||(x_k)||_p=\big(\sum_{k=1}^\infty||x_k||^p\big)^{1/p},\quad\qquad(x_k)\in\ell_s^p(X).\end{align*}
Following Sinha and Karn \cite{SK02}, the \emph{$p$-convex hull} of a weakly $p$-summable sequence $(x_k)\in\ell_w^p(X)$ is defined as
\[p\text{-co}\{x_k\}:=\big\{\sum_{k=1}^\infty\lambda_k x_k\mid (\lambda_k)\in B_{\ell_{p'}}\big\}\qquad\big(\text{if }p=1,\text{ then }(\lambda_k)\in B_{c_0}\big).\]
We say that the operator $T\in\mathcal L(X,Y)$ is a \emph{Sinha-Karn $p$-compact} operator, denoted $T\in\SK_p(X,Y)$, if there is a strongly $p$-summable sequence $(y_k)\in\ell_s^p(Y)$ such that 
\begin{equation}\label{231a}
T(B_X)\subset\, p\text{-co}\{y_k\}.\end{equation} 
According to \cite[Theorem 4.2]{SK02} and \cite[Proposition 3.15]{DPS10}, the class $\SK_p=(\SK_p,\V_{\SK_p})$ of the Sinha-Karn $p$-compact operators is a Banach operator ideal where the ideal norm $\V_{\SK_p}$ is defined by
\[||T||_{\SK_p}=\inf\big\{||(y_k)||_p\mid \eqref{231a}\text{ holds for }(y_k)\in\ell_s^p(Y)\big\}\]
for all $T\in\SK_p(X,Y)$. 

We point out that Sinha-Karn $p$-compact operators were introduced by Sinha and Karn \cite{SK02} as $p$-compact operators and with the notation $(K_p,\kappa_p)$ for the corresponding Banach operator ideal. However, we have decided to adopt the terminology and notation from  \cite{TW22, Wir} in order to obtain a clear distinction from the historically earlier class of (classical) $p$-compact operators \cite{FS79, Pietsch80}, which are denoted by $\K_p$ in this paper (see below).
\smallskip

For the definition of unconditionally $p$-compact operators, recall from e.g. \cite[8.2]{DF93} that a weakly $p$-summable sequence $(x_k)\in\ell_w^p(X)$ is called \emph{unconditionally $p$-summable}, denoted $(x_k)\in\ell_{u}^p(X)$, if 
\[\lim_{k\to\infty}||(0,\ldots,0,x_k,x_{k+1},\ldots)||_{p,w}=0.\]
Following J.M. Kim \cite{Kim14}, the operator $T\in\mathcal L(X,Y)$ is called \emph{unconditionally p-compact}, denoted $T\in\K_{up}(X,Y)$, if there is an unconditionally $p$-summable sequence $(y_k)\in\ell_u^p(Y)$ such that
\begin{equation}\label{f}
T(B_X)\subset \,p\text{-co}\{y_k\}.
\end{equation}
The class $\K_{up}=(\K_{up},||\cdot||_{\K_{up}})$ is a Banach operator ideal \cite[Theorem 2.1]{Kim14} where the ideal norm $\V_{\K_{up}}$ is defined by
\[||T||_{\K_{up}}=\inf\big\{||(y_k)||_{p,w}\mid (\ref{f})\text{ holds for }(y_k)\in\ell_u^p(Y)\big\}\]
for all $T\in\K_{up}(X,Y)$. 

It is known that $\K_{up}\subset\K$ for all $1\le p<\infty$, where $\K=(\K,\V)$ denotes the Banach operator ideal of the \emph{compact} operators with the uniform operator norm $\V$, see e.g. \cite[pp. 1574--1575]{AO15}. Moreover, clearly $\ell_s^p(Y)\subset\ell_u^p(Y)$ for any Banach space $Y$ and $||(y_k)||_{p,w}\le ||(y_k)||_p$ for any strongly $p$-summable sequence $(y_k)\in\ell_s^p(Y)$. Thus
\begin{equation}\label{30124}
\SK_p\subset\K_{up}\subset\K
\end{equation} 
for all $1\le p<\infty$. For the Sinha-Karn $p$-compact operators we have the following monotonicity property according to \cite[Proposition 4.3]{SK02} (see also \cite[p. 949]{Oja12}):
\begin{equation}\label{monotonicityofSKp}
\SK_p\subset\SK_q
\end{equation}
for all $1\le p<q<\infty$. We point out that the classes $\K_{up}$ are not monotone for $p\in[1,\infty)$, see Remark \ref{Remark191} below. 
\smallskip

In \cite[pp. 2885--2886]{MOP15} (see also \cite[Theorem 4.5]{Fourie18}) it was established that
\begin{equation}\label{KupisKpsur}
\K_{up}=\K_{p'}^{sur}
\end{equation}
for all $1\leq p<\infty$. Here $\K_{p'}^{sur}$ denotes the surjective hull of the $p'$-compact operators $\K_{p'}$. Our results on the $\K_{up}$-AP mostly draw on this characterisation and thus we next recall the definitions associated with the surjective hull $\K_p^{sur}$ for the convenience of the reader.
\smallskip
 
Let $1\le p\le \infty$. Following Fourie and Swart \cite{FS79} or Pietsch \cite[18.3]{Pietsch80}, the operator $T\in\mathcal L(X,Y)$ is called \emph{$p$-compact}, denoted $T\in\K_p(X,Y)$, if there are compact operators $A\in\K(X,\ell^p)$ and $B\in\K(\ell^p,Y)$ such that $T=BA$, where $\ell^p$ is replaced by $c_0$ in the case of $p=\infty$. The $p$-compact norm is defined by $||T||_{\K_p}=\inf||A||\,||B||$, where the infimum is taken over all such factorisations $T=BA$. According to \cite[Theorem 2.1]{FS79} or \cite[18.3.1]{Pietsch80}, the class $\K_{p}=(\K_p,\V_{\K_p})$ is a Banach operator ideal for all $1\le p\le \infty$.
\smallskip

The \emph{surjective hull} $\K_{p}^{sur}=(\K_p^{sur},\V_{\K_p^{sur}})$ of $\K_p$ is defined by the components
\[\K_p^{sur}(X,Y)=\{T\in\mathcal L(X,Y)\mid TQ\in\K_p(\ell^1(B_X),Y)\},\]
and the associated ideal norm is defined by $||T||_{\K_p^{sur}}=||TQ||_{\K_p}$ for all $T\in \K_p^{sur}(X,Y)$. Here $Q:\ell^1(B_X)\to X$ is the canonical metric surjection. 
\smallskip

Recall further that the \emph{injective hull} $\K_p^{inj}=(\K_p^{inj},\V_{\K_p^{inj}})$ of $\K_p$ is defined by the components 
\[\K_p^{inj}(X,Y)=\{T\in\mathcal L(X,Y)\mid j_0T\in\K_p(X,\ell^\infty(B_{Y^*}))\},\]
where $j_0:Y\to \ell^\infty(B_{Y^*})$ is the canonical isometric embedding. The associated ideal norm is defined by $||T||_{\K_p^{inj}}=||j_0T||_{\K_p}$ for all $T\in\K_p^{inj}(X,Y)$.
\smallskip

We proceed by revisiting a classical compact factorisation result due to Johnson \cite{Joh71}. The novelty here is a careful verification of the equality of the respective operator ideal norms, which is used in the proof of Proposition \ref{inclusionofKS} below. Note that Proposition \ref{proposition1} can also be deduced from \cite[Corollary 2.4]{FS79} for closed subspaces $Y\subset\ell^p$. However, we indicate a proof for closed subspaces $Y\subset L^p(\mu)$ for an arbitrary measure space $(\Omega,\Sigma,\mu)$, since we will require the case $L^p[0,1]$ in Proposition \ref{inclusionofKS} as well as the case $\ell^p$ later on.
\begin{proposition}\label{proposition1}
Suppose that $1\le p<\infty$ and let $(\Omega,\Sigma,\mu)$ be a measure space. Then for all Banach spaces $X$ and all closed subspaces $Y\subset L^p(\mu)$ the following hold: \[\K(X,Y)=\K_p^{inj}(X,Y)\] 
and $||T||=||T||_{\K_p^{inj}}$ for all compact operators $T\in\K(X,Y)$. 
\end{proposition}
\begin{proof}
Let $T\in\K(X,Y)$, where $X$ is an arbitrary Banach space and $Y$ is a closed subspace of $L^p(\mu)$. Let $j:Y\to L^p(\mu)$ denote the inclusion map so that $jT\in\K(X,L^p(\mu))$. It is well known that $L^p(\mu)$ has the AP, and thus there is a sequence $(F_n)\subset\F(X,L^p(\mu))$ such that \[||F_n-jT||\to 0\] 
as $n\to\infty$, see \cite[Theorem 1.e.4]{LT77}. Moreover, $L^p(\mu)$ is an $\mathcal L_{p,\lambda}$-space for all $\lambda>1$ (see e.g. \cite[Theorem 3.2]{DJT95}), and thus the proof of \cite[Theorem 2]{Joh71} yields that $||F_n||_{\K_p}\le \lambda||F_n||$ for all $\lambda>1$ and all $n\in\mathbb N$. It follows that 
\[||F_n||_{\K_p}=||F_n||\quad(n\in\mathbb N)\]
by (B1).  Thus, by the completeness of the operator ideal norm $||\cdot||_{\K_p}$ and (B1), we have $jT\in\K_p(X,L^p(\mu))$ and 
\[||F_n-jT||_{\K_p}\to 0\] 
as $n\to\infty$. It follows that
\begin{equation}\label{192222}
||jT||_{\K_p}=\lim_{n\to\infty} ||F_n||_{\K_p}=\lim_{n\to\infty}||F_n||=||jT||=||T||.
\end{equation}
Finally, since $jT\in\K_p(X,L^p(\mu))$ we have $T\in\K_p^{inj}(X,Y)$ and $||T||_{\K_p^{inj}}\le ||jT||_{\K_p}$, see \cite[Proposition 8.4.4]{Pietsch80}. Thus, by applying \eqref{192222} and (B1) one gets that $||T||=||T||_{\K_p^{inj}}$.
\end{proof}
 
We next collect from \cite[Section 3]{Fourie18} (see also \cite[Section 2]{Fourie83}) the following pertinent characterisation of the injective and the surjective hull of the $p$-compact operators as well as the precise relationship between these two classes of operators for convenient reference. For part (iv) below recall that the \emph{dual ideal} $\I^{dual}=(\I^{dual},\V_{\I^{dual}})$ of the Banach operator ideal $\I=(\I,\VI)$ is defined by the components 
\[\I^{dual}(X,Y)=\{T\in\mathcal L(X,Y)\mid T^*\in\I(Y^*,X^*)\},\]
and the associated ideal norm is defined by $||T||_{\I^{dual}}=||T^*||_{\I}$ for all $T\in\I^{dual}(X,Y)$.

\begin{fact}\label{281}
Let $X$ and $Y$ be Banach spaces and let $T\in\mathcal L(X,Y)$. Suppose that $1\le p<\infty$ and $1<q\le \infty$. Here $\ell^q$ is replaced by $c_0$ in the case of $q=\infty$ in part (ii) and part (iii).
\begin{enumerate}[(i)]
\item By \cite[Proposition 3.4]{Fourie18} the following holds: $T\in\K_p^{inj}(X,Y)$ if and only if there is a closed subspace $M\subset\ell^p$ together with compact operators $A\in\K(X,M)$ and $B\in\K(M,Y)$ such that 
\[T=BA\quad\text{ and }\quad||T||_{\K_p^{inj}}=\inf ||A||\,||B||,\]
where the infimum is taken over all such factorisations $T=BA$. 

\item By \cite[Proposition 3.7]{Fourie18} the following holds: $T\in\K_{q}^{sur}(X,Y)$ if and only if there is a quotient space $Z$ of $\ell^q$ and compact operators $A\in\K(X,Z)$ and $B\in\K(Z,Y)$ such that 
\begin{equation}\label{192}
T=BA\quad\text{ and }\quad||T||_{\K_q^{sur}}=\inf ||A||\,||B||,
\end{equation}
where the infimum is taken over all such factorisations $T=BA$. 

\item By \cite[Proposition 3.7.(c)]{Fourie18} it suffices to assume that $A$ in \eqref{192} is a bounded operator. In particular, if $Z$ is quotient space of $\ell^q$ then (by choosing $A$ in \eqref{192} to be the identity operator $I:Z\to Z$) one gets that
\begin{equation}\label{1922}
\K(Z,Y)=\K_q^{sur}(Z,Y)\quad\text{ and }\quad ||T||=||T||_{\K_q^{sur}}
\end{equation}
for all $T\in\K(Z,Y)$.

\item By \cite[Theorem 3.13]{Fourie18} the following holds: 
\[\K_p^{inj}=(\K_{p'}^{sur})^{dual}\quad\text{ and }\quad\K_{p'}^{sur}=(\K_p^{inj})^{dual}.\]
\end{enumerate}
\end{fact}

We proceed by establishing the following reverse monotonicity property for the classes $\K_p^{inj}$ and $\K_{up}$ in the range $[1,2]$, which is used in the proof of Theorem \ref{theorem1}. Here the range cannot be extended to $[1,q]$ for any $q>2$ as we note in Remark \ref{Remark191} below.
\begin{proposition}\label{inclusionofKS}
Suppose that $1\leq p< q\leq 2$. Then the following hold:
\begin{enumerate}[(i)]
\item $\K_q^{inj}\subset\K_p^{inj}$ . 
\item $\K_{uq}\subset\K_{up}$ .
\end{enumerate}
\end{proposition}
\begin{proof}
(i) Let $\varepsilon>0$ and suppose that $T\in \K_q^{inj}(X,Y)$, where $X$ and $Y$ are arbitrary Banach spaces. By applying Fact \ref{281}.(i), there exists a factorisation $T=VU$ through a closed subspace $M\subset\ell^q$, where $U\in\K(X,M)$ and $V\in\K(M,Y)$ are compact operators such that 
\begin{equation}\label{202021a}
||V||=1\quad\text{ and }\quad ||U||<||T||_{\K_q^{inj}}+\varepsilon.
\end{equation} 
Now, $\ell^q$ embeds isometrically into $L^p[0,1]$ (see e.g. \cite[Theorem 6.4.18]{AK06}), and thus there is an isometric embedding $j:M\to L^p[0,1]$. By Proposition \ref{proposition1} we have that $U\in\K_p^{inj}(X,M)$ and 
\begin{equation}\label{20201}
||U||=||U||_{\K_p^{inj}}.
\end{equation} 
Thus $T=VU\in\K_p^{inj}(X,Y)$ by the operator ideal property. Moreover, by applying (B3) and the norm estimates in \eqref{202021a} and \eqref{20201} one gets that
\[||T||_{\K_p^{inj}}=||VU||_{\K_p^{inj}}\le ||V||\,||U||_{\K_p^{inj}}=||U||<||T||_{\K_q^{inj}}+\varepsilon.\]
Consequently, $||T||_{\K_p^{inj}}\le ||T||_{\K_q^{inj}}$. This yields the claim in part (i).
\medskip

(ii) Clearly $\I\subset\J$ implies $\I^{dual}\subset\J^{dual}$ for any Banach operator ideals $\I=(\I,\VI)$ and $\J=(\J,\V_\J)$. Thus the following identities hold by \eqref{KupisKpsur}, Fact \ref{281}.(iv) and part (i):
\[\K_{uq}=\K_{q'}^{sur}=(\K_q^{inj})^{dual}\subset (\K_p^{inj})^{dual}=\K_{p'}^{sur}=\K_{up}.\]
This completes the proof.
\end{proof}
\begin{remark}\label{Remark191}
The reverse monotonicity behaviour of the classes $\K_p^{inj}$ and $\K_{up}$ in Proposition \ref{inclusionofKS} does not extend to the case of $1\le p< 2<q<\infty$. In fact, in that case there are, according to \cite[Proposition 3.6]{TW22}, closed subspaces $X\subset\ell^p$ and $Y\subset\ell^q$ for which $\mathcal{K}_p^{inj}(X\oplus Y)$ and $\mathcal{K}_q^{inj}(X\oplus Y)$ are incomparable classes of operators. Consequently, also $\K_{up}(X^*\oplus Y^*)$ and $\K_{uq}(X^*\oplus Y^*)$ are incomparable classes by \eqref{KupisKpsur} and the first identity in Fact \ref{281}.(iv). Note that the injective hull $\K_p^{inj}$ is denoted by $\mathcal{QK}_p$ in \cite{TW22}, see the discussion after \cite[Remarks 3.2]{TW22}.
\end{remark}

Let $\I=(\I,\VI)$ and $\J=(\J,\V_\J)$ be Banach operator ideals. Recall that the quasi-Banach operator ideal $\J\circ \I=(\J\circ\I,\V_{\J\circ \I})$ is defined as follows: the operator $T\in\J\circ\I(X,Y)$ by definition if there is a Banach space $Z$ and operators $A\in\I(X,Z)$ and $B\in\J(Z,Y)$ such that $T=BA$.
The components $\J\circ\I(X,Y)$ are equipped with the quasi-norm 
\[||T||_{\J\circ\I}=\inf ||A||_{\I}\,||B||_{\J},\] 
where the infimum is taken over all such factorisations $T=BA$. We refer to \cite[6.1.1]{Pietsch80} or \cite[9.2]{DF93} for the definition of a quasi-norm. 
\smallskip

We will require the following useful characterisation of the unconditionally $p$-compact operators obtained by J.M. Kim \cite[Theorem 2.2]{Kim17}:
\begin{equation}\label{KupisKupKup} 
\K_{up}=\K_{up}\circ\K_{up}\quad\text{ for all }1\le p<\infty.
\end{equation}
\begin{remark}
The argument in \cite{Kim17} uses a technical lemma on certain collections of summable sequences of positive real numbers. We note in passing that one can also establish the isometric identity in \eqref{KupisKupKup} by applying the characterisation $\K_{up}=\K_{p'}^{sur}$ from \eqref{KupisKpsur} and using the identities in \eqref{192} and \eqref{1922} above. We leave the details to the reader.
\end{remark}

\subsection{Approximation properties}
Suppose that $\I=(\I,\VI)$ is an arbitrary Banach operator ideal and let $X$ be a Banach space. Recall that $X$ is said to have the \emph{$\mathcal I$-approximation property} ($\mathcal I$-AP) if 
\[
\I(Y,X)=\overline{\F(Y,X)}^{\V_\mathcal I}
\]
for every Banach space $Y$. Recall further that $X$ has the \emph{approximation property} (AP) if for all $\varepsilon>0$ and all compact subsets $K\subset X$ there is a bounded finite-rank operator $U\in\F(X)$ such that
\[\sup_{x\in K}||Ux-x||<\varepsilon.\]
A classical characterisation of Grothendieck (see e.g. \cite[Theorem 1.e.4]{LT77}) states that $X$ has the AP if and only if $\K(Y,X)=\A(Y,X)$ for every Banach space $Y$. Here $\A=(\A,\V)$ denotes the Banach operator ideal of the \emph{approximable} operators, which is defined by the components $\A(X,Y)=\overline{\F(X,Y)}^{\V}$ and equipped with the uniform operator norm $\V$. Consequently, $X$ has the AP if and only if $X$ has the $\K$-AP. 

Following the terminology in \cite{TW22}, we say that the Banach space $X$ has the \emph{uniform $\I$-approximation property} (uniform $\I$-AP) if 
\[
\I(Y,X)\subset\A(Y,X)
\] 
for every Banach space $Y$. The uniform $\I$-AP was considered by Lassalle and Turco \cite{LT13} with a slightly different terminology. Observe that if $X$ has the $\I$-AP, then $X$ has the uniform $\I$-AP, since $||\cdot||\le ||\cdot||_{\I}$ by (B1). However, in general the converse fails, see \cite[p. 2460]{LT13}.
\smallskip

We proceed with a simple lemma which exhibits some connections between the approximation properties presented above in special situations that are relevant for us.
\begin{lemma}\label{2876}
Suppose that $\I=(\I,\VI)$ and $\J=(\J,\V_{\J})$ are Banach operator ideals and let $X$ be a Banach space. 
\begin{enumerate}[(i)] 
\item Suppose that $\I(Y,X)=\J\circ\I(Y,X)$ for every Banach space $Y$. If $X$ has the uniform $\J$-AP, then $X$ has the $\I$-AP.

\item Suppose that $\I(Y,X)=\K\circ\I(Y,X)$ for every Banach space $Y$. If $X$ has the AP, then $X$ has the $\I$-AP.
\end{enumerate}
\end{lemma}

\begin{proof}
(i) Suppose that $X$ has the uniform $\J$-AP. Let $\varepsilon>0$ and suppose that $T\in\I(Y,X)$ for an arbitrary Banach space $Y$. By assumption, $T=BA$ for compatible operators $A\in \I(Y,Z)$ and $B\in\J(Z,X)$. Since $X$ has the uniform $\J$-AP, there is a bounded finite-rank operator $F\in\F(Z,X)$ such that $||B-F||<\varepsilon/||A||_\I$. It follows that
\[||T-FA||_\I=||BA-FA||_\I\le||B-F||\,||A||_\I<\varepsilon\]
by (B3). Consequently, $T\in \overline{\F(Y,X)}^{\VI}$ which concludes the proof.
\smallskip

(ii) Suppose that $X$ has the AP. Grothendieck's result \cite[Theorem 1.e.4]{LT77} implies that $X$ has the uniform $\K$-AP. Thus part (i) (with $\J=\K$) yields that $X$ has the $\I$-AP.
\end{proof}

Suppose that $1\le p<\infty$ and let $X$ be a Banach space. It is known that if  $X$ has the AP, then $X$ has the $\SK_p$-AP, see \cite[Proposition 3.10]{GLT12}. Moreover, using an internal  characterisation of the $\K_{up}$-AP that involves a locally convex topology on $\mathcal L(X)$, it was established in \cite[Section 2]{Kim15} that if $X$ has the AP, then $X$ has the $\K_{up}$-AP.  We next observe that this can also be verified by applying Lemma \ref{2876}.(ii) and \eqref{KupisKupKup}. Here the method is essentially the same as in the proof of \cite[Proposition 3.10]{GLT12} and for convenience we also provide the proof for the case of the $\SK_p$-AP. 

\begin{fact}\label{facts2612}  
Suppose that $X$ is a Banach space with the AP. Then $X$ has both the $\SK_p$-AP and the $\K_{up}$-AP for all $1\le p<\infty$.
\end{fact}
\begin{proof}
Let $1\le p<\infty$ and let $Y$ be an arbitrary Banach space. It follows from \cite[Theorem 3.1]{CK10} and the operator ideal property that $\SK_p(Y,X)=\K\circ\SK_p(Y,X)$. Moreover, since $\K_{up}\subset\K$ (see \eqref{30124}), the factorisation result \eqref{KupisKupKup} of Kim yields that $\K_{up}(Y,X)=\K\circ\K_{up}(Y,X)$. Thus both claims follow from Lemma \ref{2876}.(ii). 
\end{proof}

\begin{remark}\label{remarkSK2}
Every Banach space has the $\SK_2$-AP by \cite[Corollary 3.6]{DPS10_2} (see also Oja \cite[p. 952]{Oja12}) and the $\K_{u2}$-AP by \cite[Corollary 1.2]{Kim15}. The argument for the $\K_{u2}$-AP in \cite{Kim15} relies on a duality result between the $\SK_p$-AP and the $\K_{up}$-AP, see \eqref{eq:SKpimpliesKup} below. We note that a similar argument to the one suggested by Oja in \cite{Oja12} for the $\SK_2$-AP also yields the case of the $\K_{u2}$-AP. In fact, by applying the isometric identity $\K_{u2}=\K_{2}^{sur}$ from \eqref{KupisKpsur}, one verifies that $\K_{u2}=\K_2$, see \cite[18.3 and 18.1.8]{Pietsch80} or \cite[Remark 4.2]{MOP15}. Moreover, it is known that $\K_2(Y,X)=\overline{\F(Y,X)}^{\V_{\K_2}}$ for all Banach spaces $Y$ and $X$, which follows from e.g. \cite[18.1.4]{Pietsch80}. 
\end{remark}
It follows from \cite[Theorem 2.1]{DOPS09} that the uniform $\SK_p$-AP coincides with the $p$-approximation property of Sinha and Karn introduced in \cite{SK02}, which is a strictly weaker property than the $\SK_p$-AP at least for $1\le p<2$, see \cite[Theorem 2.4]{DPS10_2}. We next note that the identity $\K_{up}=\K_{up}\circ\K_{up}$ of J.M. Kim in \eqref{KupisKupKup} implies that $\K_{up}$-AP is equivalent to the uniform $\K_{up}$-AP for all $1\le p<\infty$. This also follows by combining \eqref{KupisKpsur} and \cite[Proposition 4.4]{Kim19}, but we provide for convenience a short proof by applying Lemma \ref{2876}.(i).

\begin{proposition}\label{KupAPisuniform} Suppose that $1\le p<\infty$ and let $X$ be a Banach space. Then $X$ has the $\K_{up}$-AP if and only if $X$ has the uniform $\K_{up}$-AP.
\end{proposition}
\begin{proof}
The forward implication clearly holds. The converse implication follows from Lemma \ref{2876}.(i) since $\K_{up}=\K_{up}\circ\K_{up}$ by \eqref{KupisKupKup}. 
\end{proof}

We conclude this section with a characterisation of the $\K_{up}$-AP for dual Banach spaces, which complements the characterisations of the $\SK_p$-AP and the uniform $\SK_p$-AP for dual spaces obtained in \cite[Theorem 2.3]{DPS10_2} and \cite[Theorem 2.8]{DOPS09}, respectively.
\begin{theorem}\label{28761}
Suppose that $1\le p<\infty$ and let $X$ be a Banach space. Then  the following are equivalent:
\begin{enumerate}
\item[(i)] $X^*$ has the $\K_{up}$-AP.
\item[(ii)] $\K(Z,X^*)=\overline{\F(Z,X^*)}^{\V_{\K_{up}}}$ for every quotient space $Z$ of $\ell^{p'}$ (respectively, of $c_0$ if $p=1$).
\item[(iii)] $\K(Z,X^*)=\A(Z,X^*)$ for every quotient space $Z$ of $\ell^{p'}$ (respectively, of $c_0$ if $p=1$).
\item[(iv)] $\K(X,M)=\A(X,M)$ for every closed subspace $M\subset\ell^p$.
\end{enumerate}
\end{theorem}
\begin{proof}
(i)$\Rightarrow$(ii) Suppose that $Z$ is quotient space of $\ell^{p'}$ (respectively, of $c_0$ if $p=1$) and let $T\in\K(Z,X^*)$. By Fact \ref{281}.(iii) and \eqref{KupisKpsur}, we have $T\in\K_{up}(Z,X^*)$. It follows by the assumption that $T\in\overline{\F(Z,X^*)}^{\V_{\K_{up}}}$.
\smallskip

(ii)$\Rightarrow$(iii) This is obvious since $||\cdot||\le||\cdot||_{\K_{up}}$ by (B1).
\smallskip

(iii)$\Rightarrow$(iv) Suppose that $T\in \K(X,M)$, where $M$ is an arbitrary closed subspace of $\ell^p$. By Proposition \ref{proposition1} we have $T\in\K_p^{inj}(X,M)$, and thus $T^*\in\K_{p'}^{sur}(M^*,X^*)$ by Fact \ref{281}.(iv). By Fact \ref{281}.(ii) there is a quotient space $Z$ of $\ell^{p'}$ (respectively, of $c_0$ if $p=1$) and compact operators $A\in\K(M^*,Z)$, $B\in\K(Z,X^*)$ such that $T^*=BA$. By assumption, $B\in\A(Z,X^*)$ and thus $T^*\in\A(M^*,X^*)$ by the operator ideal property. It then follows from the principle of local reflexivity that $T\in\A(X,M)$, see e.g. \cite[Theorem 11.7.4]{Pietsch80}. 
\smallskip

(iv)$\Rightarrow$(i) Suppose $T\in\K_{up}(Y,X^*)$ for an arbitrary Banach space $Y$. By \eqref{KupisKpsur} and Fact \ref{281}.(iv) we have $T^*\in\K_p^{inj}(X^{**},Y^*)$. Thus, according to Fact \ref{281}.(i), there is a closed subspace $M\subset\ell^p$ and compact operators $A\in\K(X^{**},M)$ and $B\in\K(M,Y^*)$ such that $T^*=BA$. 

Next, let $j_X:X\to X^{**}$ and $j_Y:Y\to Y^{**}$ denote the canonical isometric embeddings. By assumption $Aj_X\in \K(X,M)=\A(X,M)$, and thus $(Aj_X)^*\in\A(M^*,X^*)$. It follows that \begin{equation*}\label{122022}
T=(Aj_X)^*B^*j_Y\in \A(Y,X^*)
\end{equation*}
by the operator ideal property. Consequently, $X^*$ has the uniform $\K_{up}$-AP and thus the $\K_{up}$-AP by Proposition \ref{KupAPisuniform}.
\end{proof}

\section{The \texorpdfstring{$\mathcal K_{\MakeLowercase{up}}$}{Kup}-approximation property in the case of \texorpdfstring{$1\le \MakeLowercase{p}<2$}{1 leq p<2}}\label{section3}

In this section we establish the main results of this paper concerning the $\K_{up}$-AP for $1\le p<2$. As an application, we provide a negative answer to the case $p=1$ of questions of J.M. Kim \cite[Problem 1]{Kim17_2}; namely, the $\K_{u1}$-AP implies neither the $\SK_1$-AP nor the AP, and the $\SK_1$-AP implies neither the $\K_{u1}$-AP nor the AP. We also provide an answer to a query of Kim posed in \cite[Example 4.4]{Kim20}; namely, the $\mathcal W_1$-AP does not imply the AP, where $\mathcal W_1=(\mathcal W_1,\V_{\mathcal W_1})$ denotes the Banach operator ideal of the weakly $1$-compact operators \cite{SK02}.
\smallskip

Our first main result follows from results established in Section \ref{section2}.
\begin{theorem}\label{theorem1}
Suppose that $1\leq p< q< 2$ and let $X$ be a Banach space. If $X$ has the $\K_{up}$-AP, then $X$ has the $\K_{uq}$-AP.
\end{theorem}
\begin{proof}
Suppose that $X$ has the $\K_{up}$-AP and let $Y$ be an arbitrary Banach space. It follows by the assumption that $X$ has the uniform $\K_{up}$-AP, and thus
\[\K_{uq}(Y,X)\subset\K_{up}(Y,X)\subset\A(Y,X),\]
where the former inclusion holds by Proposition \ref{inclusionofKS}.(ii). This shows that $X$ has the uniform $\K_{uq}$-AP, and thus $X$ has the $\K_{uq}$-AP by Proposition \ref{KupAPisuniform}.
\end{proof}

Recall from \cite[0.2]{Pisier86} or \cite[31.5]{DF93} that the bounded operator $T\in\mathcal L(X,Y)$ is called \emph{compactly approximable}, denoted $T\in\mathcal{CA}(X,Y)$, if for all $\varepsilon>0$ and all compact subsets $K\subset X$ there is a bounded finite-rank operator $U\in \F(X,Y)$ such that \[\sup_{x\in K}||Ux-Tx||<\varepsilon.\]
The class $\mathcal{CA}=(\mathcal{CA},||\cdot||)$ is a Banach operator ideal equipped with the uniform operator norm $\V$. Furthermore, if $X$ or $Y$ has the AP, then 
\begin{equation}\label{3112}
\mathcal{CA}(X,Y)=\mathcal L(X,Y).
\end{equation} 
We refer to \cite[Proposition 4.1]{TW22} for a proof of these facts.
\smallskip

We proceed with a result which involves a factorisation result of compactly approximable operators due to Pisier \cite{Pisier80} and a lemma of John \cite{Jo} on the approximation property for reflexive spaces. This complements a similar result obtained in \cite[Theorem 2.2]{TW22}, which in turn combines a factorisation result of Kwapien and Maurey with John's lemma. See also \cite[Theorem 2.2]{Godefroy} and its subsequent comments as well as the remark on \cite[p. 248]{DJT95} for other results in this direction. Recall that $\ell^p$ and all closed subspaces $M\subset\ell^p$ have cotype 2 whenever $1\le p\le 2$ and type 2 whenever $2\le p<\infty$, see e.g. \cite[Theorem 6.2.14 and Remark 6.2.11.($f$)]{AK06}.
\begin{theorem}\label{Pisier-John}
Let $X$ and $Y$ be Banach spaces. Suppose that $X^*$ has cotype 2 and that $Y$ has cotype 2 and the AP. Then \begin{equation*}
\K(X,M)=\A(X,M)\end{equation*}
for every closed subspace $M\subset Y$.
\end{theorem}
\begin{proof}
Let $M$ be a closed subspace of $Y$ and suppose that $T\in\K(X,M)$. Since $Y$ has the AP, it follows from \eqref{3112} that $jT\in\CA(X,Y)$, where $j:M\to Y$ is the inclusion map. Furthermore, since $X^*$ and $Y$ have cotype 2, a factorisation result of Pisier \cite[Corollaire 2.16]{Pisier80} (see also \cite[Theorem 4.1]{Pisier86}) yields a factorisation \[jT=VU,\] where  $U\in\mathcal L(X,H)$ and $V=\mathcal L(H,Y)$ are compatible bounded operators and $H$ is a Hilbert space. It follows that $T=\widetilde V\widetilde U$, where $\widetilde U:X\to \overline{U(X)}$ and $\widetilde V:\overline{U(X)}\to M$ are the bounded operators defined as follows:
\[\widetilde U:x\mapsto Ux\quad\text{ and }\quad \widetilde V:h\mapsto Vh.\]
Now, since $\overline{U(X)}\subset H$ is reflexive and has the AP, we have $T\in\A(X,M)$ according to \cite[Remarks 3]{Jo} (see also the proof of \cite[Theorem 2.2]{TW22} for a more direct argument). Consequently, $\K(X,M)=\A(X,M)$.
\end{proof}

By applying Theorem \ref{Pisier-John} we obtain the following result, which exhibits a large class of Banach spaces that have the $\K_{up}$-AP for all $1\le p<2$. For an alternative proof using a different method, see Remark \ref{0201} below.

\begin{theorem}\label{theorem2}
Suppose that $X$ is a Banach space with cotype 2. Then $X$ has the $\K_{up}$-AP for all $1\le p< 2$.
\end{theorem}
\begin{proof}
Suppose that $1\le p<2$ and let $T\in\K_{up}(Y,X)$, where $Y$ is an arbitrary Banach space. By Proposition \ref{KupAPisuniform} it suffices to verify that $T\in\A(Y,X)$. 

Towards this, note that $T^{*}\in\K_{p}^{inj}(X^*,Y^*)$ by \eqref{KupisKpsur} and Fact \ref{281}.(iv). Thus there is, according to Fact \ref{281}.(i), a closed subspace $M\subset\ell^p$ and compact operators $A\in\K(X^*,M)$ and $B\in\K(M,Y^*)$ such that 
\[T^*=BA.\] 
Since $X$ has cotype 2, the bidual $X^{**}$ has cotype 2, see e.g. \cite[Corollary 11.9]{DJT95}. Thus, by applying Theorem \ref{Pisier-John}, one gets that 
\[A\in\K(X^*,M)=\A(X^*,M).\] 
It follows from the operator ideal property that $T^*\in\A(X^*,Y^*)$, and thus $T\in\A(Y,X)$, see e.g. \cite[Theorem 11.7.4]{Pietsch80}. This completes the proof.
\end{proof}
\begin{remark}\label{0201}
Theorem \ref{theorem2} can also be established by an approach which involves the duality between the $\K_{up}$-AP and the $\SK_p$-AP due to J.M. Kim \cite{Kim15} and Lassalle and Turco \cite{LT17}. In fact, suppose that $1\le p<2$ and let $X$ be a Banach space with cotype 2. By applying a classical result of Maurey on absolutely $p$-summing operators and a characterisation of the $\SK_p$-AP for dual spaces due to Delgado et al. \cite{DPS10_2}, it is shown in \cite[Proposition 3.8]{Wir} that $X^*$ has the $\SK_p$-AP. Consequently, $X$ has the $\K_{up}$-AP by \cite[Theorem 1.1]{Kim15} if $1<p<2$ and by \cite[Theorem 4.7]{LT17} if $p=1$.
\end{remark}
The following example shows that the monotonicity property in Theorem \ref{theorem1} does not extend to the case of $1\le p<2<q<\infty$. 
\begin{eexample}\label{example}
Suppose that $1\le p< 2<q<\infty$. Then there is a reflexive Banach space that has the $\K_{up}$-AP but fails the $\K_{uq}$-AP. In fact, let $X\subset\ell^q$ be a closed subspace that fails the AP, see \cite[Theorem 2.d.6]{LT77}. Then $X$ fails the $\SK_q$-AP according to \cite[Theorem 1]{Oja12}, and thus $X^*$ fails the $\K_{uq}$-AP by \cite[Theorem 1.1]{Kim15} (see \eqref{eq:KupimpliesSKp} below). Moreover, since $X$ has type 2, the dual space $X^*$ has cotype 2, see e.g. \cite[Proposition 11.10]{DJT95}. Consequently, $X^*$ has the $\K_{up}$-AP by Theorem \ref{theorem2}.
\end{eexample}

In \cite[Section 5]{Kim17_2} J.M. Kim discusses the relationship between the AP, the $\SK_p$-AP and the $\K_{up}$-AP. In particular, the author demonstrates for all $1<p<2$ that the $\K_{up}$-AP implies neither the $\SK_p$-AP nor the AP, and vice versa, the $\SK_p$-AP implies neither the $\K_{up}$-AP nor the AP.  Whether or not the same holds for $p=1$ or for $2<p<\infty$ was stated as a problem \cite[Problem 1]{Kim17_2}.
 \smallskip
 
The reasoning above by Kim relies among others on results on the approximation property of order $p$, whose study was initiated by Saphar \cite{Saphar70}, and a duality result of Kim. In fact, according to \cite[Theorem 1.1]{Kim15}, the following claims hold for all $1<p<\infty$ and all Banach spaces $X$:
\begin{align}
\label{eq:KupimpliesSKp}&\text{If }X^*\text{ has the }\K_{up}\text{-AP, then }X\text{ has the }\SK_{p}\text{-AP.}\\
\label{eq:SKpimpliesKup}&\text{If } X^*\text{ has the }\SK_p\text{-AP, then }X\text{ has the }\K_{up}\text{-AP.}
\end{align} 
Subsequently it was shown by Kim \cite[Theorem 1.1]{Kim17_2} that \eqref{eq:KupimpliesSKp} also holds for $p=1$, and by Lassalle and Turco \cite[Theorem 4.7]{LT17} that also the claim \eqref{eq:SKpimpliesKup} holds for $p=1$. See also \cite[Theorem 2.5]{KLT21} for a general result that includes both statements \eqref{eq:KupimpliesSKp} and \eqref{eq:SKpimpliesKup} for $p=1$. 
\smallskip

We will next demonstrate that Theorem \ref{theorem1} and Theorem \ref{theorem2} together with these duality results yield an answer to the case $p=1$ of \cite[Problem 1]{Kim17_2}, where the following questions were posed by J.M. Kim (with our notation): 
\begin{itemize}
\item[(Q1)]\emph{If $X$ has the $\K_{u1}$-AP (respectively, the $\SK_1$-AP), then does $X$ have the $\SK_1$-AP (respectively, the $\K_{u1}$-AP) or the AP ?}
\end{itemize}
The analogous questions were also asked for $2<p<\infty$ and we will in Section 5 draw attention to a few related questions and remarks that could potentially be of use for the case $2<p<\infty$.
 
We also obtain a negative answer to the case $p=1$ of the follow-up questions \cite[Problem 2]{Kim17_2} of Kim:
\begin{itemize}
\item[(Q2)]\emph{If $X^*$ has the $\K_{u1}$-AP (respectively, the $\SK_1$-AP), then does $X$ have the $\K_{u1}$-AP (respectively, the $\SK_1$-AP) ?}
\end{itemize}

Recall for any $1\le p<2$ that there is a closed subspace $X\subset\ell^p$ that fails the AP due to Szankowski, see \cite{Sza78} or \cite[Theorem 1.g.4]{LT79}.

\begin{proposition}\label{answer1}
Suppose that $1<q<2$ and let $X\subset\ell^q$ be a closed subspace that fails the AP. Then $X$ has the $\K_{u1}$-AP but fails the $\SK_1$-AP. Furthermore, the dual space $X^*$ has the $\SK_1$-AP but fails both the AP and the $\K_{u1}$-AP. 

In particular, the $\K_{u1}$-AP implies neither the $\SK_1$-AP nor the AP, and the $\SK_1$-AP implies neither the $\K_{u1}$-AP nor the AP. This answers (Q1) above in the negative. Moreover, note that the fact that the closed subspace $X\subset\ell^q$ in Proposition \ref{answer1} is reflexive yields a negative answer to (Q2) above.
\end{proposition}
\begin{proof}
Since $X$ has cotype 2, it has the $\K_{u1}$-AP by Theorem \ref{theorem2}. Consequently, since $X$ is reflexive, the dual space $X^*$ has the $\SK_1$-AP according to \cite[Theorem 1.1]{Kim17_2}.

Next, we show that $X^*$ fails the $\K_{u1}$-AP. For this, assume towards a contradiction that $X^*$ has the $\K_{u1}$-AP. Then $X^*$ has the $\K_{uq}$-AP by Theorem \ref{theorem1}. Consequently, $X$ has the $\SK_q$-AP by the duality \eqref{eq:KupimpliesSKp}. But since $X\subset\ell^q$ is a closed subspace, it follows from \cite[Theorem 1]{Oja12} that $X$ has the AP, which is a contradiction. Thus $X^*$ fails the $\K_{u1}$-AP. 

Since $X^*$ is reflexive and fails the $\K_{u1}$-AP, the duality result \cite[Theorem 1.3]{Kim17_2} or \cite[Theorem 4.7]{LT17} yields that $X$ fails the $\SK_1$-AP. Moreover, $X^*$ fails the AP by Fact \ref{facts2612} or \cite[Theorem 1.e.7]{LT77}.
\end{proof}

It is also interesting to observe that even the combination of the $\K_{up}$-AP and the $\SK_p$-AP for all $1\le p<2$ does not imply the AP  as our following example shows.
 \begin{eexample}
Let $P$ denote the Banach space that fails the AP constructed by Pisier in \cite{Pisier83} (see also \cite[Section 10]{Pisier86}), which has the following properties: 
\begin{itemize}
\item $P\widehat\otimes_\pi P=P\widehat\otimes_\varepsilon P$.
\item Both $P$ and $P^*$ have cotype 2. 
\end{itemize}
Now, $P$ has the $\K_{up}$-AP for all $1\le p< 2$ by Theorem \ref{theorem2}. The same is true for $P^*$ and thus the duality result \eqref{eq:KupimpliesSKp} of Kim, i.e. \cite[Theorem 1.1]{Kim15} for $p>1$ and \cite[Theorem 1.1]{Kim17_2} for $p=1$, yields that $P$ also has the $\SK_p$-AP for all $1\le p< 2$. 
\end{eexample}
As a final observation of this section we note that Theorem \ref{theorem2} also yields an answer to a query posed in \cite[Example 4.4]{Kim20} for $p=1$; namely, the $\mathcal W_1$-AP does not imply the AP. Here $\mathcal W_1=(\mathcal W_1,\V_{\mathcal W_1})$ denotes the Banach operator ideal of the weakly 1-compact operators, which is a special case of the weakly $p$-compact operators $\mathcal W_p=(\mathcal W_p,\V_{\mathcal W_p})$ by Sinha and Karn \cite{SK02}. Recall by definition that the operator $T\in\mathcal W_1(X,Y)$ if there is a weakly 1-summable sequence $(y_k)\in\ell_w^1(Y)$ such that
\begin{equation}\label{1710}
T(B_X)\subset\,1\text{-co}\{y_k\}= \{\sum_{k=1}^\infty\lambda_k y_k\mid(\lambda_k)\in B_{c_0}\}.\end{equation}
The ideal norm $\V_{\mathcal W_1}$ is defined by \[||T||_{\mathcal W_1}=\inf\{||(y_k)||_{1,w}\mid \eqref{1710}\text{ holds for }(y_k)\in\ell_w^1(Y)\}\]
for all $T\in\mathcal W_1(X,Y)$, see \cite[p. 864]{Kim20}. Note that \begin{equation}\label{Ku1subset}
\K_{u1}\subset\mathcal W_1\end{equation} 
since $\ell^1_u(X)\subset\ell^1_w(X)$ for any Banach space $X$.

\begin{proposition} 
Suppose that $X$ is a reflexive Banach space with cotype 2. Then $X$ has the $\mathcal W_1$-AP. 

In particular, the $\mathcal W_1$-AP does not imply the AP, which answers the query posed in \cite[Example 4.4]{Kim20}.
\end{proposition}
\begin{proof} We will apply a similar argument as in \cite[Example 4.1]{Kim20}. For this, suppose that $T\in\mathcal W_1(Y,X)$ for an arbitrary Banach space $Y$. Then there is a weakly 1-summable sequence $(x_k)\in\ell_w^1(X)$ such that $T(B_Y)\subset 1\text{-co}\{x_k\}$.
 Since $X$ is reflexive, a classical result of Bessaga and Pe\l{}czy\'{n}ski yields that $(x_k)\in\ell_u^1(X)$, see e.g. \cite[Theorem 2.4.11]{AK06}, and thus $T\in\K_{u1}(Y,X)$. Moreover, since $X$ has the $\K_{u1}$-AP by Theorem \ref{theorem2}, we have that $T\in\overline{\F(Y,X)}^{\V_{\K_{u1}}}$. It follows that $T\in\overline{\F(Y,X)}^{\V_{\mathcal W_1}}$ since $||\cdot||_{\mathcal W_1}\le ||\cdot ||_{\mathcal K_{u1}}$ by \eqref{Ku1subset}. Consequently, $X$ has the $\mathcal W_1$-AP.
\end{proof}
\section{The failure of the \texorpdfstring{$\mathcal{SK}_\MakeLowercase{p}$}{SKp}-AP for \texorpdfstring{$\MakeLowercase{p}>2$}{p>2}}\label{section4}

Let $1\le p<\infty$ and let $X$ be a Banach space. Recall that a subset $K\subset X$ is called \emph{relatively p-compact} if $K\subset p\text{-co}\{x_k\}$ for a strongly $p$-summable sequence $(x_k)\in\ell_s^p(X)$, see \cite[pp. 19--20]{SK02}. Following Sinha and Karn \cite[Section 6]{SK02}, the Banach space $X$ is said to have the \emph{$p$-approximation property} ($p$-AP) if for all $\varepsilon>0$ and all relatively $p$-compact subsets $K\subset X$ there is a bounded finite-rank operator $U\in\F(X)$ such that 
\[\sup_{x\in K}||Ux-x||<\varepsilon.\] 
If $X$ has the AP, then $X$ has the $p$-AP for all $1\le p<\infty$ since relatively $p$-compact subsets are relatively compact, see \cite[p. 20]{SK02}. Moreover, it follows from \cite[Theorem 2.1]{DOPS09} that the Banach space $X$ has the $p$-AP if and only if $X$ has the uniform $\SK_p$-AP.

Let $2<p,q<\infty$ be such that $q>2p/(p-2)$. In \cite[Corollary 2.9]{CK10} Y.S. Choi and J.M. Kim established that a variant of the closed subspace $E\subset\ell^q$ that fails the AP constructed by Davie \cite{Da73} fails the $p$-AP, i.e. the uniform $\SK_p$-AP. It follows from the monotonicity property \eqref{monotonicityofSKp} that $E$ fails the uniform $\SK_r$-AP for all $r\ge p$. In this section we complement the result of Choi and Kim in the setting of the $\SK_p$-AP by showing that for all $2<p,q<\infty$ there is a closed subspace $X\subset\ell^q$ that fails the $\SK_r$-AP for all $r\ge p$. In particular, we do not require the condition $q>2p/(p-2)$ in our example. We also point out here that it appears unknown whether the AP, the $\SK_p$-AP and the uniform $\SK_{p}$-AP are different properties for $2<p<\infty$, see Section \ref{section5}. 

The closed subspace $X\subset\ell^q$ we uncover is also essentially the subspace by Davie. However, our approach is based on a factorisation argument of Reinov \cite[Lemma 1.1]{Reinov82} in contrast to the approach of Choi and Kim in \cite{CK10}. For other similar constructions based on \cite[Lemma 1.1]{Reinov82} we refer to \cite[Theorem 3.9]{TW22} and \cite[Proposition 4.3]{Wir}.

Suppose that $1\le p<\infty$ and let $X$ and $Y$ be arbitrary Banach spaces. Recall from \cite[18.1 and 18.2]{Pietsch80} or \cite[pp. 111-112]{DJT95} that the operator $T\in\mathcal L(X,Y)$ is called $p$-nuclear, denoted $T\in\mathcal N_p(X,Y)$, if there is a strongly $p$-summable sequence $(x_k^*)\in\ell_s^p(X^*)$ and a weakly $p'$-summable sequence $(y_k)\in \ell_w^{p'}(Y)$ such that 
\begin{equation}\label{pnuclear}
Tx=\sum_{k=1}^\infty x_k^*(x)y_k,\qquad x\in X.
\end{equation}
The class $\mathcal N_p=(\mathcal N_p,\V_{\mathcal N_p})$ is a Banach operator ideal where the $p$-nuclear norm $||\cdot||_{\mathcal N_p}$ is defined by 
\[||T||_{\Nu_p}=\inf \big\{||(x^*_k)||_p\,||(y_k)||_{p',w}\mid \eqref{pnuclear}\text{ holds}\,\big\}\]
for all $T\in\Nu_p(X,Y)$.

Following Persson and Pietsch \cite{PP69}, the operator $T\in\mathcal L(X,Y)$ is called \emph{quasi $p$-nuclear}, denoted $T\in\QN_p(X,Y)$, if there is a strongly $p$-summable sequence $(x_k^*)\in\ell_s^p(X^*)$ such that 
\begin{equation}\label{3}
||Tx||\le (\sum_{k=1}^\infty |x_k^*(x)|^p)^{1/p},\quad x\in X.
\end{equation}
The class $\QN_p=(\QN_p,\V_{\QN_p})$ is a Banach operator ideal where the ideal norm $\V_{\QN_p}$ is defined by 
\[||T||_{\QN_p}=\inf \{||(x_k^*)||_p\mid \eqref{3}\text{ holds for }(x^*_k)\in\ell_s^p(X^*)\}\]
for all $T\in\QN_p(X,Y)$.
 
It is known that $\QN_p$ coincides with the injective hull of $\mathcal N_p$, see \cite[Satz 38 and Satz 39] {PP69} or the comment before \cite[Theorem 6]{Pietsch14}. Moreover, it follows from results in \cite[Section 3]{DPS10} (see also the discussion preceding \cite[Corollary 2.7]{GLT12}) that
\begin{equation}\label{duality}
\QN_p=\SK_p^{dual}
\end{equation}
for all $1\le p<\infty$.
\begin{proposition}
Let $2<p,q<\infty$. Then there is a closed subspace $X\subset\ell^q$ that fails the $\SK_r$-AP for all $r\ge p$.
\end{proposition}
\begin{proof}
Let $A=(a_{k,j})_{k,j\in\mathbb N}$ be an infinite matrix of scalars with the following properties: \begin{enumerate}[(i)]
\item $A^2=0$, \item tr $A:=\sum_{k=1}^\infty a_{k,k}\neq 0$, \item $\sum_{k=1}^\infty \lambda_k^b<\infty$ for all $b>2/3$, where $\lambda_k:=\sup_{j\in\mathbb N}|a_{k,j}|>0$ for all $k\in\mathbb N$.
\end{enumerate}
Recall that such a matrix exists due to Davie \cite{Da73,Da75} (see also \cite[Theorem 2.d.3]{LT77}). Since $2<p,q<\infty$ we may pick a positive real number $c>0$ such that \[0<c<\min\{\frac{1}{3}-\frac{2}{3p},\frac{1}{3}-\frac{2}{3q}\}.\] 
Denote $s:=\frac{2}{3q}+c$ and let $B=(b_{k,j})_{k,j\in\mathbb N}$ be the matrix with elements $b_{k,j}=\big(\frac{\lambda_j}{\lambda_k}\big)^s a_{k,j}$ for all $k,j\in\mathbb N$. By (i) and (ii), the matrix $B$ has the following properties:
\begin{enumerate}
\item[(iv)] $B^2=0$ and
\item[(v)] tr $B=\sum_{k=1}^\infty b_{k,k}\neq 0$.
\end{enumerate}
(Note that the matrix $B$ is defined as the matrix in \cite[Theorem 2.d.6]{LT77} but with a different exponent $s$.) 

Since $c>0$ we have $sq>2/3$. Consequently,
\begin{align*}
\Big(\sum_{j=1}^\infty &|b_{k,j}|^q\Big)^{\frac{1}{q}}=\Big(\sum_{j=1}^\infty \Big(\frac{\lambda_j}{\lambda_k}\Big)^{sq}|a_{k,j}|^q\Big)^{\frac{1}{q}}\\
&\leq \Big(\sum_{j=1}^\infty \Big(\frac{\lambda_j}{\lambda_k}\Big)^{sq}\lambda_k^q\Big)^{\frac{1}{q}}=\lambda_k^{1-s}\Big(\sum_{j=1}^\infty \lambda_j^{sq}\Big)^{\frac{1}{q}}<\infty
\end{align*}
for all $k\in\mathbb N$, where the last expression is finite by (iii). It follows that $y_k:=(b_{k,1},b_{k,2},\ldots)\in\ell^q$ for all $k\in\mathbb N$, and moreover,
\begin{equation}\label{normofyk}
||y_k||_q\le L\cdot\lambda_k^t \qquad(k\in\mathbb N),
\end{equation}
where $L:=\big(\sum_{j=1}^\infty\lambda_j^{sq}\big)^{\frac{1}{q}}<\infty$ and $t:=1-s$.
\smallskip

\emph{Claim.} The closed linear span 
\[X:=\overline{\text{span}}\{y_k\mid k\in\N\}\subset\ell^q\] 
fails the $\SK_r$-AP for all $r\geq p$.
\smallskip

Towards this, fix $r\in\mathbb R$ such that $r\geq p$. In order to establish the claim, we will next define relevant spaces and operators which are illustrated in \eqref{commuting} below. Since 
\begin{equation}\label{tis1}
t=1-s=1-\frac{2}{3q}-c>1-\frac{2}{3q}-\Big(\frac{1}{3}-\frac{2}{3q}\Big)=\frac{2}{3}\,,
\end{equation} we have 
\begin{equation}\label{y_ksummable}\sum_{k=1}^\infty ||y_k||_q\leq L\cdot\sum_{k=1}^\infty \lambda_k^t<\infty
\end{equation}
by (\ref{normofyk}) and (iii). It follows that the matrix $B$ defines a 1-nuclear operator $B:\ell^{q'}\to\ell^{q'}$ defined by
\[Bx=\big(\langle x,y_k\rangle\big)_{k=1}^\infty,\qquad x=(x_k)\in\ell^{q'}\]
for which (iv) and (v) hold. This means that $Bx=\sum_{k=1}^\infty y_k(x)e_k$ for all $x\in\ell^{q'}$, where $(e_k)$ denotes the unit vector basis in $\ell^{q'}$. 

Define the following operators:
\begin{align*}
&V:\ell^{q'}\to \ell^\infty, \quad Vx=\big(\lambda_k^{-t}\langle x,y_k\rangle\big)_{k=1}^\infty, \quad x=(x_k)\in\ell^{q'},\\
&\Delta:\ell^\infty\to\ell^{q'},\quad \Delta x=\big(\lambda_k^t x_k\big)_{k=1}^\infty, \qquad x=(x_k)\in\ell^\infty.
\end{align*}
By \eqref{normofyk} the operator $V$ is bounded (with $||V||\le L$). Furthermore, $\Delta$ is a 1-nuclear diagonal operator by (iii) since $t>2/3$ by \eqref{tis1}. Clearly $B=\Delta V$. 

Let $E:=\ell^{q'}/\ker V$. Note that $\ker V=X^\perp$ since $\lambda_k\neq 0$ for all $k\in\mathbb N$. Thus
\begin{equation}\label{HahnBanach}
X=X^{\perp\perp}\cong (\ell^{q'}/X^\perp)^*= E^*.
\end{equation} 
Denote $d:=\min\{r',q'\}>1$ and pick $a>0$ such that 
\begin{equation}\label{ais}
\frac{2}{3r}<a<t-\frac{2}{3d}.
\end{equation} 
This is possible since 
\[t-\frac{2}{3r'}>\frac{2}{3}-\frac{2}{3r'}=\frac{2}{3r}\]
by \eqref{tis1}, and since
\[t-\frac{2}{3q'}=1-s-\frac{2}{3q'}=1-\Big(\frac{2}{3q}+c\Big)-\frac{2}{3q'}=\frac{1}{3}-c>\frac{2}{3p}\ge\frac{2}{3r},\] 
where the last inequality holds since $r\ge p$. 

Next, define the following diagonal operators:
\begin{align*}
&\Delta_1:\ell^\infty\to \ell^r,\quad \Delta_1 x=\big(\lambda_k^a x_k\big)_{k=1}^\infty,\quad x=(x_k)\in\ell^\infty,\\
&\Delta_2: \ell^r\to\ell^{q'},\quad \Delta_2 x=\big(\lambda_k^{t-a} x_k\big)_{k=1}^\infty,\quad x=(x_k)\in\ell^r.
\end{align*}
It follows from \eqref{ais} that $ar>2/3$, and thus $(\lambda_k^a)\in \ell^r$ by (iii). Hence $\Delta_1$ is $r$-nuclear. Similarly we have $(\lambda_k^{t-a})\in\ell^d$, and thus $\Delta_2$ is $d$-nuclear. Clearly $\Delta=\Delta_2\Delta_1$. 

Consider the closed subspace $F_0:=\overline{V\ell^{q'}}\subset\ell^\infty$ and let $j_1:F_0\hookrightarrow\ell^\infty$ denote the inclusion map. Let $Q:\ell^{q'}\to E$ be the canonical projection onto $E=\ell^{q'}/\ker V$ and let $\widetilde V:E\to F_0$ be the injective operator induced by $V$ such that $V=j_1\widetilde VQ$. 

Furthermore, define the closed subspace $F:=\overline{\Delta_1 j_1 F_0}\subset\ell^r$ and let $j_2:F\hookrightarrow\ell^r$ denote the inclusion map. Define also the operator 
\[\widetilde\Delta_1:F_0\to F,\quad \widetilde\Delta_1 x=\Delta_1 j_1x,\quad x\in F_0.\]
This means that $j_2\widetilde\Delta_1=\Delta_1j_1$. Finally, let $U:=\widetilde\Delta_1\widetilde V:E\to F$. 

The above defined spaces and operators are illustrated in the following commuting diagram:
\begin{equation}\label{commuting}
\begin{tikzcd}
\ell^{q'} \arrow{rr}{V} \arrow{d}{Q} &&\ell^\infty\arrow{rd}{\Delta_1} \arrow{r}{\Delta} &\ell^{q'}\\
E\arrow[bend right=35,swap]{rr}{U}\arrow{r}{\widetilde V} & F_0 \arrow[ru,hook,"j_1"]\arrow{r}{\widetilde{\Delta}_1} & F \arrow[r,hook,"j_2"] &\ell^r\arrow[swap]{u}{\Delta_2}\\
\end{tikzcd} 
\end{equation}
We proceed by verifying that 
\begin{equation}\label{Unotin}
U\in\QN_r(E,F) \setminus\overline{\F(E,F)}^{\V_{\QN_r}}.
\end{equation} 
Firstly, since $\Delta_1$ is $r$-nuclear, the operator $\tilde\Delta_1$ is quasi $r$-nuclear by \cite[Satz 39]{PP69}. Hence $U\in \QN_r(E,F)$ by the operator ideal property.
We continue by showing that $U\notin \overline{\F(E,F)}^{\V_{\QN_r}}$.

For this, define
\begin{equation}\label{phi}
\phi(T)=\text{tr}(\Delta_2j_2TQ)
\end{equation}
for every $T\in\QN_r(E,F)$. The mapping $T\mapsto \phi(T)$ defines a continuous linear functional $\phi\in\QN_r(E,F)^*$. In fact, since $d=\min\{r',q'\}\le r'$ we have 
\[\Delta_2\in\mathcal N_{d}(\ell^r,\ell^{q'})\subset\mathcal N_{r'}(\ell^r,\ell^{q'})\]
by monotonicity, see e.g. \cite[Corollary 5.24]{DJT95}. Thus by the multiplication table in \text{\cite[Satz 48]{PP69}} the following hold: $\Delta_2j_2TQ\in\mathcal N_1(\ell^{q'})$ and
\begin{equation}\label{129}
||\Delta_2j_2TQ||_{\mathcal N_1}\leq ||\Delta_2||_{\mathcal N_{r'}}||j_2TQ||_{\mathcal{QN}_{r}} 
\end{equation} 
for all $T\in\QN_r(E,F)$. Since $\ell^{q'}$ has the (metric) AP, it follows that the trace on the right hand side in \eqref{phi} is well-defined and 
\begin{equation}\label{2010}
|\phi(T)|=|\text{tr}(\Delta_2j_2TQ)|\le ||\Delta_2j_2TQ||_{\Nu_1}
\end{equation}
for all $T\in\QN_r(E,F)$, see e.g. \cite[10.3.2]{Pietsch80}. By combining the norm estimates \eqref{129}, \eqref{2010} and (B3), one obtains that
\[|\phi(T)|\le ||\Delta_2||_{\Nu_{r'}}||j_2TQ||_{\QN_r}\le ||\Delta_2||_{\Nu_{r'}}||T||_{\QN_r}\]
for all $T\in\QN_r(E,F)$. Thus the mapping $T\mapsto \phi(T)$ defines a continuous linear functional $\phi\in\QN_r(E,F)^*$.

Next, note that
\begin{equation}\label{phi_2}
\phi(U)=\text{tr}(\Delta_2j_2UQ)=\text{tr} (\Delta V)=\text{tr}(B)\neq 0,
\end{equation}
where the final step holds by (v). However, we claim that 
\begin{equation}\label{finiterank}
\phi(x^*\otimes y)=0
\end{equation} 
for all bounded rank-one operators $x^*\otimes y\in \F(E,F)$, which by linearity and continuity implies that $\phi(T)=0$ for all $T\in\overline{\F(E,F)}^{\V_{\QN_r}}$. This will show that $U\notin \overline{\F(E,F)}^{\V_{\QN_r}}$, since $\phi(U)\neq 0$ by \eqref{phi_2}.

Towards the claim \eqref{finiterank}, let $x^*\in E^*$ and $y\in F$. By construction $F=\overline{UQ\ell^{q'}}$ and we assume first that $y\in UQ\ell^{q'}$. Thus $y=UQx$ for some $x\in\ell^{q'}$. It then follows that
\begin{align}
\label{30121}\phi(x^*\otimes y)&=\text{tr}(\Delta_2j_2(x^*\otimes y)Q)=\text{tr}(Q^*x^*\otimes \Delta_2j_2y)\\
\notag&=\langle \Delta_2 j_2 y,Q^*x^*\rangle
=\langle Q\Delta_2j_2y,x^*\rangle\\
\notag&=\langle Q\Delta_2j_2UQx,x^*\rangle=\langle QBx,x^*\rangle=0,
\end{align}
where the last equality above holds since $QB=0$. In fact, $\Delta j_1\widetilde VQB=B^2=0$ by (iv), and since $\Delta j_1 \widetilde V$ is injective, we have $QB=0$.

For an arbitrary $y\in F$, pick a sequence $(z_n)\subset UQ\ell^{q'}$ such that $||z_n-y||\to 0$ as $n\to\infty$. Then 
\[
||x^*\otimes y-x^*\otimes z_n||_{\QN_r}=||x^*\otimes (y-z_n)||_{\QN_r}
=||x^*||\,||z_n-y||\to 0
\]
as $n\to\infty$ by ideal norm property (B2). Consequently, by continuity and \eqref{30121}, we have that
\[\phi(x^*\otimes y)=\lim_{n\to\infty}\phi(x^*\otimes z_n)=0,\]
which then yields \eqref{finiterank}. Thus $U\notin\overline{\F(E,F)}^{\V_{\QN_p}}$.

We have thus established the claim in \eqref{Unotin}. It then follows by the duality \eqref{duality} and reflexivity of the spaces $E$ and $F$ that 
\[U^*\in\SK_r(F^*,E^*)\setminus\overline{\F(F^*,E^*)}^{\V_{\SK_r}}.\] 
Consequently, $E^*$ fails the $\SK_r$-AP and thus $X$ fails the $\SK_r$-AP since $X\cong E^*$ by \eqref{HahnBanach}.
\end{proof}

\section{Final remarks and problems}\label{section5}
It is unknown whether the AP, the $\K_{up}$-AP and the $\SK_p$-AP  are distinct 
properties whenever $2<p<\infty$. More precisely, in \cite[Problem 1]{Kim17_2} J.M. Kim posed the following question (with our notation and excluding the case $p=1$ which was answered in the negative in Propostion \ref{answer1}):

\begin{question}\label{question1}
(J.M. Kim) Let $2<p<\infty$. If $X$ has the $\K_{up}$-AP (respectively, the $\SK_p$-AP), then does $X$ have the $\SK_p$-AP (respectively, the $\K_{up}$-AP) or the AP ? 
\end{question}
In this section we draw attention to a few related questions and remarks suggested by our results. The first of our questions is motivated by Theorem \ref{Pisier-John}. For a related question, see \cite[Question 5.2]{TW22}.
\begin{question}\label{qu:3012}
Let $2<p<\infty$. Is there a Banach space $X$ such that $X^*$ fails the AP and $\K(X,M)=\A(X,M)$ for every closed subspace $M\subset\ell^p$ ? We note that for such a Banach space $X$, the dual $X^*$ would have the $\K_{up}$-AP by Theorem \ref{28761}.
\end{question}
It is known that there are Banach spaces that have the uniform $\SK_p$-AP but fail the $\SK_p$-AP whenever $1\le p<2$, see \cite[p. 2460]{LT13}. In fact, \emph{every} Banach space has the uniform $\SK_p$-AP for $1\le p<2$ (see \cite[Theorem 6.4]{SK02} or \cite[Corollary 2.5]{DOPS09}), and by  \cite[Theorem 2.4]{DPS10_2} there are Banach spaces that fail the $\SK_p$-AP. However, it appears unknown whether there are Banach spaces that have the uniform $\SK_p$-AP but fail the $\SK_p$-AP for $2<p<\infty$.

Recall the following result \cite[Theorem 1]{Oja12} due to Oja which has been used in Example \ref{example} and Proposition \ref{answer1}: If $X$ is a closed subspace of $L^p(\mu)$, then $X$ has the AP if and only if $X$ has the $\SK_p$-AP. 

Oja's result suggests the following question on the relationship between the $\SK_p$-AP and the uniform $\SK_p$-AP.
\begin{question}\label{qu1}
Let $2<p<\infty$. Is there a significant class $\mathcal C$ of Banach spaces for which the $\SK_p$-AP is equivalent to the uniform $\SK_p$-AP for all Banach spaces $X\in\mathcal C$ ?
\end{question}
By a significant class above we mean (somewhat vaguely) a class which is defined independently of any approximation properties and is large enough to contain spaces that have the AP and spaces that fail the uniform $\SK_p$-AP. Note that the (uniform) $\K_{up}$-AP implies the uniform $\SK_p$-AP since $\SK_p\subset\K_{up}$, see \eqref{30124}. Consequently, the $\K_{up}$-AP implies the $\SK_p$-AP in such a class $\mathcal C$ of Banach spaces.

Let $2<p<\infty$ be fixed. Lemma \ref{2876}.(i) points to the following class of Banach spaces which seems relevant for Question \ref{qu1} at first sight: let $\mathcal C'$ denote the class of all Banach spaces $X$ for which \[\SK_p(Y,X)=\SK_p\circ\SK_p(Y,X)\]for every Banach space $Y$. In fact, Lemma \ref{2876}.(i) implies that if $X\in\mathcal C'$, then $X$ has the $\SK_p$-AP if and only if $X$ has the uniform $\SK_p$-AP. However, Example \ref{lastexample} below shows that the class $\mathcal C'$ only contains the finite-dimensional Banach spaces, and is thus not a relevant class here.
\begin{eexample}\label{lastexample}
Suppose that $2<p<\infty$ and let $X$ be an arbitrary infinite-dimensional Banach space. We claim that
\begin{equation}\label{claim}
\SK_p\circ\SK_p(\ell^{p'},X)\subsetneq \SK_p(\ell^{p'},X).
\end{equation}
In fact, since $p>2$ and $X$ is infinite-dimensional, there is a relatively $p$-compact subset $K\subset X$ that is not relatively $(p/2)$-compact according to \cite[Proposition 3.5.(3)]{PD11}, see also \cite[Proposition 20.(3)]{Pietsch14}. By the definition we may assume that $K=p$-co$\{x_k\}$ for some strongly $p$-summable sequence $(x_k)\in\ell_s^p(X)$. Let $\theta:\ell^{p'}\to X$ be the bounded operator defined as follows: 
\[\theta(a)=\sum_{k=1}^\infty a_kx_k,\quad a=(a_k)\in\ell^{p'}.\] 
This means that $K=\theta(B_{\ell^{p'}})$, and thus $\theta\in\SK_p(\ell^{p'},X)$. We next verify that $\theta\notin\SK_p\circ\SK_p(\ell^{p'},X)$. 

Assume towards a contradiction that $\theta\in\SK_p\circ\SK_p(\ell^{p'},X)$. Let $\theta=VU$ be a factorisation where $U,V\in\SK_p$ are compatible Sinha-Karn $p$-compact operators. According to the duality result \cite[Proposition 3.8]{DPS10}, the adjoint operators $U^*,V^*\in\QN_p$ are quasi $p$-nuclear, and thus $\theta^*\in \QN_p\circ\QN_p(X^*,\ell^p)$. 

Next, according to the multiplication rule \cite[Satz 48]{PP69} we have $\theta^*\in\QN_{p/2}(X^*,\ell^p)$, and thus $\theta\in\SK_{p/2}(\ell^{p'},X)$ by \cite[Proposition 3.8]{DPS10}. This means that $\theta(B_{\ell_{p'}})=K\subset X$ is relatively $(p/2)$-compact, which is a contradiction. We have thus shown that 
\[\theta\in\SK_p(\ell^{p'},X)\setminus\SK_p\circ \SK_p(\ell^{p'},X),\] 
which yields the claim \eqref{claim}.
\end{eexample}

As a final remark we recall that it is a long standing open question whether the Hardy space $H^\infty$ of the bounded analytic functions on the open unit disc has the AP, see \cite[Problem 1.e.10]{LT77} or \cite[5.2(4)]{DF93}. The connection here to Question \ref{question1} by Kim is the following observation: 
\[H^\infty\text{ has the }\K_{up}\text{-AP for all }1<p<\infty.\]
This follows from a similar argument as in \cite[Corollary 2.10]{DOPS09} where it was shown that all odd duals $(H^\infty)^*$, $(H^\infty)^{***}$, $\ldots$ have the uniform $\SK_p$-AP for all $1\le p<\infty$. In fact, fix $1<p<\infty$ for the argument. Recall from e.g. \cite[Section 1]{DOPS09} or \cite[21.7]{DF93} that the Banach space $X$ has the \emph{approximation property of order p} (AP$_p$) if the natural mapping 
\[\theta: Y^*\widehat\otimes_{g_p} X\to \Nu_p(Y,X)\]
is injective for all Banach spaces $Y$. Here $g_p$ denotes the Chevet-Saphar tensor norm and we refer to e.g. \cite[12.7]{DF93} for a description of $g_p$. A result of Bourgain and Reinov \cite[Theorem 1]{BR85} implies that the bidual $(H^\infty)^{**}$ has the AP$_{p'}$. It then follows that $(H^\infty)^*$ has the $\SK_p$-AP by \cite[Corollary 2.5]{DPS10_2}. Consequently, $H^\infty$ has the $\K_{up}$-AP by \cite[Theorem 1.1]{Kim15} (see \eqref{eq:SKpimpliesKup} above).

\section*{Acknowledgements} 
This paper is part of the Ph.D.-thesis of the author and he thanks his supervisor Hans-Olav Tylli for many valuable comments and remarks during the preparation of the manuscript. The author also thanks Pablo Turco for supplying reference \cite{KLT21} and for helpful comments.

This work was supported by the Magnus Ehrnrooth Foundation and the Swedish Cultural Foundation in Finland.

\bibliographystyle{plain}

\bibliography{bibliographyWirzenius}
\end{document}